\documentclass[12pt]{amsart}

\usepackage{amsmath,amssymb,amsthm}
\usepackage{mathtools}
\usepackage{enumerate}
\usepackage{tikz}
\usepackage[dvipsnames]{xcolor}

\usepackage[margin=1.2in]{geometry}
\usepackage{hyperref}
\hypersetup{
  colorlinks=true,
  linkcolor=NavyBlue,
  citecolor=OliveGreen,
  urlcolor=MidnightBlue
}

\theoremstyle{plain}
\newtheorem{theorem}{Theorem}[section]
\newtheorem{lemma}[theorem]{Lemma}

\theoremstyle{definition}
\newtheorem{definition}[theorem]{Definition}
\newtheorem{example}[theorem]{Example}

\theoremstyle{remark}
\newtheorem{remark}[theorem]{Remark}

\newcommand{\LatAut}{\operatorname{LatAut}}
\newcommand{\sgn}{\operatorname{sgn}}
\newcommand{\Aut}{\operatorname{Aut}}

\newcommand{\id}{\operatorname{id}}
\newcommand{\GL}{\operatorname{GL}}

\begin{document}

\title[Termination of the LatAut Tower of Symmetric Groups]{Termination of the Lattice-Automorphism Tower\\
for Direct Products of Symmetric Groups}

\author[]{Sonukumar}
\address{Department of Mathematics, Manipal Institute of Technology}
\email{sonupnkumar@gmail.com}

\author[]{Vinay Madhusudanan}
\address{Department of Mathematics, Manipal Institute of Technology}
\email{vinay.m2000@gmail.com}

\date{\today}



\begin{abstract}
Let $G$ be a finite group. Let $\mathcal{N}(G)$ be the lattice of normal
subgroups ordered by inclusion, regarded as an abstract lattice.  Define 
$\LatAut(G) := \Aut(\mathcal{N}(G))$. The \emph{LatAut tower} is the
sequence defined by $G_0 = G$, $G_{n+1} = \LatAut(G_n)$.

Let $G$ be  a \emph{tower group} if $G \cong \prod_{k \geq 3} S_k^{a_k}$
with finitely many $a_k \neq 0$. We establish the following for tower groups.

\smallskip
\noindent\emph{Product Formula.}\;
$\LatAut\!\bigl(\prod_{k \geq 3} S_k^{a_k}\bigr) \cong S_{a_4} \times S_B$,
where $B = \sum_{k \geq 3,\, k \neq 4} a_k$.

\smallskip
\noindent\emph{Termination Theorem.}\; For every finite group $G$, 
$G_3 = 1$ for every tower group $G_0$, and we also prove that this bound is sharp with an example. 

\smallskip
The proof applies Goursat's lemma to classify $\mathcal{N}(G)$ into three
families parameterised by admissible triples $(J,\mathbf{P},H)$ as 
sub-products, sign-parity elements, and mixed elements and uses the
Krull--Schmidt theorem to identify the direct factors $S_k^{(k,i)}$ as precisely the nontrivial
indecomposable complemented elements of $\mathcal{N}(G)$
(the complemented elements being exactly the full sub-products).
These results do not extend to groups outside the tower-group family. 
\end{abstract}

\maketitle

\section{Introduction}

The \emph{automorphism tower} of a finite group $G$ is obtained by iterating
the ordinary automorphism group:
\[
  G \to \Aut(G) \to \Aut(\Aut(G)) \to \cdots .
\]
Wielandt~\cite{Wielandt1939} proved that this
tower terminates for every finite centerless group, and Thomas~\cite{Thomas1985}
extended the analysis to the transfinitely iterated case. The analogous
question for lattice automorphisms is structurally different. $\LatAut(G)$ is
determined solely by the order-type of $\mathcal{N}(G)$, discarding all
group-theoretic information beyond normal subgroup inclusion. For a rich family
of groups this compression is dramatic, and termination can be established
concretely with a sharp step count.

We work throughout with \emph{tower groups} which are finite direct products of symmetric groups
$G = \prod_{k \geq 3} S_k^{a_k}$, where $S_k$ denotes the symmetric group
on $k$ letters and finitely many exponents $a_k \geq 0$ are nonzero. Our
main results are:

\begin{theorem}[Product Formula]\label{thm:product}
Let $G = \prod_{k \geq 3} S_k^{a_k}$ be a tower group. Then
\[
  \LatAut(G) \;\cong\; S_{a_4} \times S_B,
  \qquad B = \sum_{\substack{k \geq 3 \\ k \neq 4}} a_k.
\]
\end{theorem}

\begin{theorem}[Termination]\label{thm:termination}
For any tower group $G_0$, the LatAut tower satisfies $G_3 = 1$.
This bound is sharp: $G_0 = S_4^2 \times S_3^2$ gives $G_3 = 1$ with $G_2 \neq 1$.
\end{theorem}

The proof of termination has four steps. In Section~\ref{sec:lattice} we give a complete
description of $\mathcal{N}(G)$ via a bijection with \emph{admissible triples}
$(J,\mathbf{P},H)$, where $J \subseteq \{(k,i) : k \geq 3,\; a_k > 0,\; 1 \leq i \leq a_k\}$
records which individual copy-slots participate in sign couplings,
$\mathbf{P} = (P_{k,i})_{(k,i) \notin J}$ records the local normal
subgroup data for uncoupled factors, and $H \leq \{\pm 1\}^J$ records
the joint sign constraint.
This parameterization encompasses three families:
\emph{sub-product elements} $\prod_{k,i} P_{k,i}^{(k,i)}$ (corresponding to
$J = \emptyset$), \emph{sign-parity elements} $D_I$ (corresponding to $J = I$
with $P_{k,i} = S_k$ for $(k,i) \notin I$ and $H = H_I$ the product-one kernel),
and \emph{mixed elements}. All remaining admissible triples, characterised by
$J \neq \emptyset$ and either $H \neq H_J$ or $P_{k,i} \neq S_k$ for some
$(k,i) \notin J$.
In Section~\ref{sec:injection} we show that every lattice automorphism
permutes the set of direct factors $\{S_k^{(k,i)}\}$. The complemented
elements of $\mathcal{N}(G)$ are exactly the full sub-products
$\prod_{(k,i) \in T} S_k^{(k,i)}$; among these the nontrivial
indecomposable ones are precisely the individual factors $S_k^{(k,i)}$,
which are therefore preserved as a set by every lattice automorphism.
This gives an injection $\LatAut(G) \hookrightarrow S_{a_4} \times S_B$.
Section~\ref{sec:realisation} proves surjectivity by constructing, for each
$\sigma \in \operatorname{Sym}(\mathcal{A}) \times \operatorname{Sym}(\mathcal{B})$,
an explicit lattice automorphism $\tau_\sigma \in \LatAut(G)$ defined directly
on admissible triples via the unique chain isomorphism between equal-length
chains. Order preservation is verified using the Inclusion Characterization
(Lemma~\ref{lem:incl-char}), requiring no separate case analysis for
cross-degree transpositions.
Section~\ref{sec:termination} uses the Product Formula
to carry out the termination analysis by complete case enumeration.
Section~\ref{sec:examples} works through explicit examples.

\section{Notation and Preliminaries}

Throughout the paper, all groups are finite. We write $\mathcal{N}(G)$ for the lattice
of normal subgroups of $G$, with meet $H \wedge K = H \cap K$ and
join $H \vee K = HK$. This lattice is modular by a theorem of
Ore~\cite{Ore1938}; see also Birkhoff~\cite{Birkhoff1967}.

\begin{definition}
The \emph{lattice-automorphism group} of $G$ is
\[
  \LatAut(G) \;=\; \Aut\!\bigl(\mathcal{N}(G)\bigr),
\]
the group of all automorphisms of $\mathcal{N}(G)$ regarded as an abstract
lattice. Explicitly, $\LatAut(G)$ consists of all bijections
$\varphi \colon \mathcal{N}(G) \to \mathcal{N}(G)$ satisfying
$H \leq K \Leftrightarrow \varphi(H) \leq \varphi(K)$ for all $H, K \in \mathcal{N}(G)$.

The \emph{LatAut tower} of $G$ is the sequence of finite groups
\[
  G_0 = G, \qquad G_{n+1} = \Aut\!\bigl(\mathcal{N}(G_n)\bigr) \quad (n \geq 0),
\]
where at each step one forms the normal subgroup lattice $\mathcal{N}(G_n)$
of the group $G_n$ and takes its automorphism group. The notation $\LatAut(G_n)$
is used interchangeably with $\Aut(\mathcal{N}(G_n))$ throughout.
\end{definition}

\begin{definition}
A finite group $G$ is a \emph{tower group} if
$G \cong \prod_{k \geq 3} S_k^{a_k}$ where finitely many $a_k \geq 0$ are
nonzero. We write the individual factors as $S_k^{(k,i)}$ for each $k$ with
$a_k > 0$ and $1 \leq i \leq a_k$, and $\pi_{k,i} \colon G \to S_k$ for the
projection onto the $i$-th copy of $S_k$.
\end{definition}

The following two tools are used in the proofs.

\begin{theorem}[Goursat's Lemma, {\cite[Theorem~1.6.1]{Schmidt1994}}]
\label{thm:goursat}
Let $G = H \times K$. The subgroups $N \leq G$ are in bijection with quintuples
$(H_0, H_1, K_0, K_1, \varphi)$ where $H_1 \trianglelefteq H_0 \leq H$,
$K_1 \trianglelefteq K_0 \leq K$, and $\varphi \colon H_0/H_1 \xrightarrow{\;\sim\;} K_0/K_1$.
The subgroup corresponding to this data is
\[
  N = \bigl\{(a,b) \in H_0 \times K_0 : \varphi(aH_1) = bK_1\bigr\}.
\]
The subgroup $N$ is normal in $G$ if and only if
\[
  [H,\,H_0] \leq H_1 \qquad\text{and}\qquad [K,\,K_0] \leq K_1.
\]
Equivalently, $H$ acts trivially by conjugation on $H_0/H_1$, and
$K$ acts trivially by conjugation on $K_0/K_1$.
In particular,
\[
H_0,H_1 \trianglelefteq H
\qquad\text{and}\qquad
K_0,K_1 \trianglelefteq K.
\]\end{theorem}


\begin{theorem}[Krull--Remak--Schmidt, {\cite[Theorem~3.3.8]{Robinson1996}}]
\label{thm:ks}
Let $G = H_1 \times \cdots \times H_r = K_1 \times \cdots \times K_s$
be two decompositions of $G$ into indecomposable factors. Then $r = s$
and there is a central automorphism $\alpha$ of $G$ such that, after
suitable relabelling, $H_i^\alpha = K_i$ for $i = 1, \ldots, r$.
\end{theorem}

\begin{theorem}[Remak, {\cite[Theorem~3.3.12]{Robinson1996}}]
\label{thm:remak}
Let $G = \prod_{\lambda \in \Lambda} G_\lambda$ where each $G_\lambda$
is indecomposable with $Z(G_\lambda) = 1$. Then every normal subgroup
$N$ of $G$ that has a complement in $G$ is the direct product of
certain of the $G_\lambda$.
\end{theorem}

We also use the simplicity of alternating groups.

\begin{theorem}[{\cite[Theorem~3.2.3]{Robinson1996}}]
\label{thm:alternating}
For $k \geq 5$, the alternating group $A_k$ is simple. For $k = 3$,
$A_3 \cong C_3$ is cyclic. For $k = 4$, $A_4$ is solvable with
$V_4 \trianglelefteq A_4$ and $A_4/V_4 \cong C_3$.
\end{theorem}

The normal subgroups of $S_k$ for small $k$ are as follows. For $k = 3$:
$\mathcal{N}(S_3) = \{1, A_3, S_3\}$, a chain of length~$2$. For $k = 4$:
$\mathcal{N}(S_4) = \{1, V_4, A_4, S_4\}$, a chain of length~$3$ (uniqueness of
$V_4$ as the minimal non-trivial normal subgroup follows from the fact that it is
the unique subgroup of order $4$ normal in $S_4$). For $k \geq 5$:
$\mathcal{N}(S_k) = \{1, A_k, S_k\}$, a chain of length~$2$
(since $A_k$ is the unique proper non-trivial normal subgroup by simplicity).

\section{The Normal Subgroup Lattice of a Tower Group}
\label{sec:lattice}

Fix a tower group $G = \prod_{k \geq 3} S_k^{a_k}$, with individual factors
$S_k^{(k,i)}$ for each $k$ with $a_k > 0$ and $1 \leq i \leq a_k$.

\begin{definition}\label{def:triples}
Let $N \in \mathcal{N}(G)$ define the \emph{canonical triple}
$(J(N),\,\mathbf{P}(N),\,H(N))$ by
\begin{align}
  J(N) &:= \bigl\{(k,i) : a_k > 0,\; 1 \leq i \leq a_k,\;
              \pi_{k,i}(N) = S_k\ \text{and}\
              N \cap S_k^{(k,i)} = A_k\bigr\},        \label{eq:J} \\
  P_{k,i}(N) &:= \pi_{k,i}(N) \in \mathcal{N}(S_k),
              \quad (k,i) \notin J(N),                           \label{eq:P} \\
  H(N) &:= \bigl\{(\sgn g_{k,i})_{(k,i) \in J(N)} : g \in N\bigr\}
           \;\leq\; \{\pm 1\}^{J(N)}.                       \label{eq:H}
\end{align}
A triple $(J,\,(P_{k,i})_{(k,i)\notin J},\,H)$ is called \emph{admissible} if
$J \subseteq \{(k,i) : k \geq 3,\; a_k > 0,\; 1 \leq i \leq a_k\}$,
$P_{k,i} \in \mathcal{N}(S_k)$ for $(k,i) \notin J$,
and $H \leq \{\pm1\}^J$ satisfies, for every $(k,i) \in J$:
\begin{enumerate}[(i)]
  \item\label{adm:i} the element $(-1_{(k,i)}) := $ the vector in $\{\pm1\}^J$
        with $-1$ in position $(k,i)$ and $+1$ elsewhere
        does not lie in $H$, and
  \item\label{adm:ii} some element of $H$ has $-1$ in coordinate $(k,i)$.
\end{enumerate}
For any admissible triple $(J,\mathbf{P},H)$, define
\[
  N(J,\mathbf{P},H) :=
  \bigl\{(g_{k,i})_{k,i} \in G : g_{k,i} \in P_{k,i}\ (\forall\, (k,i) \notin J),\;
  (\sgn g_{k,i})_{(k,i) \in J} \in H\bigr\}.
\]
A \emph{sub-product element} is any $N(\emptyset,\mathbf{P},\{()\})
= \prod_{(k,i)} P_{k,i}^{(k,i)}$ with $P_{k,i} \in \mathcal{N}(S_k)$,
where $\{()\}$ denotes the trivial group $\{\pm1\}^\emptyset$
containing only the empty tuple. For $I \subseteq \{(k,i) : k \geq 3,\; a_k > 0,\; 1 \leq i \leq a_k\}$ with
$|I| \geq 2$, the \emph{sign-parity element} is
$D_I := N(I,\,(S_k)_{(k,i)\notin I},\,H_I)$
where $H_I := \{(\varepsilon_{k,i})_{(k,i)\in I} \in \{\pm1\}^I :
\prod_{(k,i)\in I}\varepsilon_{k,i} = 1\}$.
Since $\sgn\colon S_k \to \{\pm 1\}$ is surjective for every $k \geq 3$,
each $D_I$ is normal in $G$ of index~$2$.
\end{definition}

\begin{lemma}[Complete classification of $\mathcal{N}(G)$]\label{lem:lattice}
The canonical-triple map $\Phi\colon N \mapsto (J(N),\mathbf{P}(N),H(N))$
is a bijection from $\mathcal{N}(G)$ onto the set of admissible triples.
The inverse sends $(J,\mathbf{P},H)$ to $N(J,\mathbf{P},H)$.

In particular, $\mathcal{N}(G)$ contains all sub-product elements and all
sign-parity elements $D_I$ ($|I| \geq 2$), together with further
\emph{mixed elements}; these three families are pairwise disjoint and
together exhaust $\mathcal{N}(G)$.
\end{lemma}

\begin{proof}
\
For any triple $(J,\mathbf{P},H)$, that is not necessarily admissible,
$N(J,\mathbf{P},H)$ is normal in $G$.
For $h = (h_{k,i})_{k,i} \in G$ and $g = (g_{k,i})_{k,i} \in N(J,\mathbf{P},H)$,
for $(k,i) \notin J$, $h_{k,i} g_{k,i} h_{k,i}^{-1} \in P_{k,i}$ since
$P_{k,i} \trianglelefteq S_k$;
for $(k,i) \in J$, $\sgn(h_{k,i} g_{k,i} h_{k,i}^{-1}) = \sgn(g_{k,i})$ since $\sgn$
is a group homomorphism. Therefore conjugation preserves both conditions.

\medskip

Fix $N \in \mathcal{N}(G)$ and any pair $(k,i)$ with $a_k > 0$, $1 \leq i \leq a_k$.
By applying Theorem~\ref{thm:goursat} to the bipartition
$G = S_k^{(k,i)} \times G_{(k,i)}'$, where
$G_{(k,i)}' = \prod_{(k',i') \neq (k,i)} S_{k'}^{(k',i')}$,
set  $N_{k,i} := \pi_{k,i}(N)$ and $M_{k,i} := N \cap S_k^{(k,i)}$.
The normality condition requires $[S_k,\,N_{k,i}] \leq M_{k,i}$.
We determine all pairs $(N_{k,i}, M_{k,i})$ with
$M_{k,i} \trianglelefteq N_{k,i}$ such that $N_{k,i} \in \mathcal{N}(S_k)$.
The trivial pairs $(N_{k,i}, N_{k,i})$ always satisfy
$[S_k, N_{k,i}] \leq N_{k,i}$ and give quotient~$1$.
Among non-trivial pairs $M_{k,i} \subsetneq N_{k,i}$:

For $k \neq 4$, where $\mathcal{N}(S_k) = \{1, A_k, S_k\}$,
the non-trivial pairs are:
\begin{itemize}
  \item $(S_k,\,A_k)$: $[S_k, S_k] = A_k \leq A_k$.
    Valid; quotient $C_2$.
  \item $(S_k,\,1)$: $[S_k,S_k] = A_k \not\leq 1$.
    Invalid.
  \item $(A_k,\,1)$: $[S_k, A_k] = A_k \not\leq 1$.
    ($[S_k, A_k] = A_k$: for $k = 3$,
    $[(12),(123)] = (123)$ generates $A_3$;
    for $k \geq 5$, $A_k$ is simple and non-abelian.)
    Invalid.
\end{itemize}

For $k = 4$, where $\mathcal{N}(S_4) = \{1, V_4, A_4, S_4\}$,
the six non-trivial pairs are:
\begin{itemize}
  \item $(S_4,\,A_4)$: $[S_4, S_4] = A_4 \leq A_4$.
    Valid; quotient $C_2$.
  \item $(S_4,\,V_4)$: $[S_4,S_4] = A_4 \not\leq V_4$
    ($A_4$ contains $3$-cycles; $V_4$ does not). Invalid.
  \item $(S_4,\,1)$: $A_4 \not\leq 1$. Invalid.
  \item $(A_4,\,V_4)$: $[S_4,A_4] = A_4 \not\leq V_4$.
    ($[S_4,A_4] = A_4$: since $[(12),(123)] = (123)$ and
    $3$-cycles generate $A_4$.) Invalid.
  \item $(A_4,\,1)$: $[S_4,A_4] = A_4 \not\leq 1$. Invalid.
  \item $(V_4,\,1)$: $[S_4,V_4] = V_4 \not\leq 1$.
    ($[S_4,V_4] = V_4$: since $[(12),(13)(24)] = (12)(34) \in V_4$
    and $V_4$ is the minimal normal subgroup of $S_4$.) Invalid.
\end{itemize}
Hence $N_{k,i}/M_{k,i} \in \{1,C_2\}$, with $N_{k,i}/M_{k,i} \cong C_2$ if and only
if $(N_{k,i}, M_{k,i}) = (S_k, A_k)$, which is precisely the condition
$(k,i) \in J(N)$.

\smallskip
\noindent\textit{Case~A: Trivial coupling when $(k,i) \notin J(N)$.}
The coupling quotient is $N_{k,i}/M_{k,i} \cong 1$, so
$\pi_{k,i}(N) = N \cap S_k^{(k,i)} = P_{k,i}(N)$.

\smallskip
\noindent\textit{Case~B: $C_2$-coupling when $(k,i) \in J(N)$.}
With $N_{k,i} = S_k$ and $M_{k,i} = A_k$, the Goursat isomorphism
$\varphi\colon N^{(k,i)}/M^{(k,i)} \xrightarrow{\;\sim\;} \{\pm 1\}$
(where $N^{(k,i)} = \pi_{G_{(k,i)}'}(N)$ and $M^{(k,i)} = N \cap G_{(k,i)}'$)
induces $\chi := \varphi \circ \pi \colon N^{(k,i)} \to \{\pm 1\}$
with $\ker\chi = M^{(k,i)}$.

\smallskip
\noindent\textit{Key Claim: $E \leq M^{(k,i)}$, where
$E := N^{(k,i)} \cap \prod_{(k',i') \neq (k,i)} A_{k'}^{(k',i')}$.}
Suppose for the sake of contradiction that $h \in E$ satisfies $\chi(h) = -1$.
Then there exists some $(g_{k,i}, h) \in N$ and has $\sgn g_{k,i} = -1$.
Since $h \in \prod_{(k',i') \neq (k,i)} A_{k'}^{(k',i')}$, every component
$h_{k',i'}$ is even.
Construct $\kappa \in N$ as follows: for $(k',i') \notin J(N)$,
$(k',i') \neq (k,i)$, set $\kappa_{k',i'} := h_{k',i'}$,
which lies in $P_{k',i'}(N) = N \cap S_{k'}^{(k',i')}$ (since
$h \in N^{(k,i)}$ forces $h_{k',i'} \in \pi_{k',i'}(N)$, and
$\pi_{k',i'}(N) = N \cap S_{k'}^{(k',i')}$ by Case~A);
for $(k',i') \in J(N)$, $(k',i') \neq (k,i)$, set $\kappa_{k',i'} := h_{k',i'}$,
which lies in $A_{k'} = N \cap S_{k'}^{(k',i')}$ since $\sgn h_{k',i'} = +1$;
and set $\kappa_{k,i} := e \in A_k = N \cap S_k^{(k,i)}$.
Since each element with $\kappa_{k',i'}$ in position $(k',i')$ and identity
elsewhere lies in $N$, and elements from distinct factors commute, $\kappa \in N$.
Since $\kappa_{k',i'} = h_{k',i'}$ for all $(k',i') \neq (k,i)$ and
$\kappa_{k,i} = e$, the product $(g_{k,i}, h) \cdot \kappa^{-1}$ in $G$ has
$(k,i)$-th coordinate $g_{k,i} \cdot e^{-1} = g_{k,i}$ and $(k',i')$-th coordinate
$h_{k',i'} \cdot h_{k',i'}^{-1} = e$ for each $(k',i') \neq (k,i)$.
Hence $(g_{k,i}, e, \ldots, e) = (g_{k,i}, h) \cdot \kappa^{-1} \in N \cap
S_k^{(k,i)} = A_k$, contradicting $\sgn g_{k,i} = -1$.
Hence $E \leq M^{(k,i)}$.

\smallskip
\noindent\textit{Character extension.}
Since $\ker(\psi|_{N^{(k,i)}}) = E \leq \ker\chi$, where
$\psi \colon N^{(k,i)} \to \{\pm 1\}^{\{(k',i') \neq (k,i)\}}$ denotes the sign map
$n \mapsto (\sgn n_{k',i'})_{(k',i') \neq (k,i)}$, the character $\chi$ factors
through $\psi(N^{(k,i)})$ via a unique
$\mathbb{F}_2$-linear map $f \colon \psi(N^{(k,i)}) \to \{\pm 1\}$
satisfying $\chi = f \circ \psi|_{N^{(k,i)}}$.

Every $\mathbb{F}_2$-linear functional on a subspace of $\{\pm 1\}^{\{(k',i') \neq (k,i)\}}$
is the restriction of a coordinate-product functional, so there exists
$L \subseteq \{(k',i') : (k',i') \neq (k,i)\}$ such that
$f(v) = \prod_{(k',i') \in L} v_{k',i'}$ for all $v \in \psi(N^{(k,i)})$.
(The set $L$ is not uniquely determined by $f$; two sets $L$ and $L'$
represent the same $f$ on $\psi(N^{(k,i)})$ if and only if
$\prod_{(k',i') \in L \triangle L'} \sgn n_{k',i'} = +1$ for every
$n \in N^{(k,i)}$; the character $\chi = f \circ \psi|_{N^{(k,i)}}$ is
independent of the choice.)
Hence
\[
  \chi(n) \;=\; \prod_{(k',i') \in L} \sgn n_{k',i'}
  \quad\text{for all } n \in N^{(k,i)},
\]
and consequently $g \in N$ if and only if $(g_{k',i'})_{(k',i') \neq (k,i)} \in N^{(k,i)}$
and $\sgn g_{k,i} = \prod_{(k',i') \in L}\sgn g_{k',i'}$.
This coupling contributes $(k,i)$ to $J(N)$ and contains the resulting joint
sign constraint within $H(N)$.
The set $J(N)$ may properly contain $J(N^{(k,i)}) \cup \{(k,i)\}$: when $\chi$
involves factors not previously sign-active in $G_{(k,i)}'$, the $C_2$-coupling
to factor $(k,i)$ activates new sign correlations among those factors, producing
\emph{mixed elements} of types not captured by the two classical families.

\medskip
Surjectivity is established in the following steps. 

Let $K := \prod_{(k,i)}(N \cap S_k^{(k,i)})$.
Since each $N \cap S_k^{(k,i)} \leq N$ and we know that elements from distinct
factors commute, $K \leq N$.

\smallskip
\noindent$(\supseteq)$\;
For $n = (n_{k,i})_{k,i} \in N$: for $(k,i) \notin J(N)$,
$n_{k,i} \in \pi_{k,i}(N) = P_{k,i}(N)$ by definition of $\pi_{k,i}$;
and $(\sgn n_{k,i})_{(k,i) \in J(N)} \in H(N)$ by~\eqref{eq:H}.
Hence $n \in N(J(N),\mathbf{P}(N),H(N))$.

\smallskip
\noindent$(\subseteq)$\;
Let $g = (g_{k,i})_{k,i} \in N(J(N),\mathbf{P}(N),H(N))$.
By~\eqref{eq:H} choose $n \in N$ with
$(\sgn n_{k,i})_{(k,i) \in J(N)} = (\sgn g_{k,i})_{(k,i) \in J(N)}$,
and set $h := gn^{-1}$.
For $(k,i) \in J(N)$: $\sgn h_{k,i} = 1$, so
$h_{k,i} \in A_k = N \cap S_k^{(k,i)}$.
For $(k,i) \notin J(N)$: $g_{k,i} \in P_{k,i}(N) = N \cap S_k^{(k,i)}$ (given)
and $n_{k,i} \in \pi_{k,i}(N) = N \cap S_k^{(k,i)}$, so
$h_{k,i} = g_{k,i} n_{k,i}^{-1} \in N \cap S_k^{(k,i)}$.
Hence $h \in K \leq N$, and $g = hn \in N$.

\medskip
Sub-claim:
The triple $\Phi(N)$ is admissible.

\noindent\textit{Condition~\ref{adm:i}.}
Suppose for the sake of contradiction that $(-1_{(k,i)}) \in H(N)$ for some
$(k,i) \in J(N)$.
Then there exists some $n \in N$ that satisfies $\sgn n_{k,i} = -1$ and
$\sgn n_{k',i'} = +1$ for all $(k',i') \in J(N)\setminus\{(k,i)\}$.
Constructing $\kappa \in K$ by setting $\kappa_{k',i'} := n_{k',i'}$ for
$(k',i') \notin J(N)$
(this is valid since $n_{k',i'} \in P_{k',i'}(N) = N \cap S_{k'}^{(k',i')}$),
$\kappa_{k',i'} := n_{k',i'}$ for $(k',i') \in J(N)\setminus\{(k,i)\}$
(valid since $\sgn n_{k',i'} = +1$ gives $n_{k',i'} \in A_{k'} = N \cap S_{k'}^{(k',i')}$),
and $\kappa_{k,i} := e$.
Then $n \kappa^{-1} = (e,\ldots,n_{k,i},\ldots,e) \in N$
with $n_{k,i} \notin A_k$, contradicting $N \cap S_k^{(k,i)} = A_k$.

\noindent\textit{Condition~\ref{adm:ii}.}
Since $(k,i) \in J(N)$ implies $\pi_{k,i}(N) = S_k$, some $g \in N$ has
$\sgn g_{k,i} = -1$, so $H(N)$ is not contained in $\{h_{(k,i)} = +1\}$.

\noindent\textit{Injectivity.}
The triple $(J(N),\mathbf{P}(N),H(N))$ is uniquely recovered from $N$
via~\eqref{eq:J}--\eqref{eq:H}.

\medskip

Given any admissible triple $(J,\mathbf{P},H)$, let $N := N(J,\mathbf{P},H)$
and we verify $\Phi(N) = (J,\mathbf{P},H)$.

\smallskip
\noindent\textit{We have, $J(N) = J$.}
For $(k,i) \in J$, an element with $g_{k,i}$ in position $(k,i)$ and identity
elsewhere lies in $N$ if and only if the sign constraint
$(\ldots,+1,\sgn g_{k,i},+1,\ldots) \in H$
(with $+1$ in every $(k',i')$-coordinate for $(k',i') \in J\setminus\{(k,i)\}$)
holds. By admissibility~\eqref{adm:i}, $(-1_{(k,i)}) \notin H$, so
$g_{k,i} \in A_k$. The identity $(+1,\ldots,+1) \in H$ permits every
$g_{k,i} \in A_k$. Hence $N \cap S_k^{(k,i)} = A_k$.
By admissibility~\eqref{adm:ii}, for some $h \in H$ has $h_{(k,i)} = -1$.
Choose $g_{k',i'} \in S_{k'}$ with $\sgn g_{k',i'} = h_{(k',i')}$ for
$(k',i') \in J$ and any $g_{k',i'} \in P_{k',i'}$ for $(k',i') \notin J$;
this gives $g \in N$ with $\sgn g_{k,i} = -1$, so $\pi_{k,i}(N) = S_k$.
Hence $(k,i) \in J(N)$.
For $(k,i) \notin J$: the element with $g_{k,i}$ in position $(k,i)$ and
identity elsewhere lies in $N$ if and only if
$g_{k,i} \in P_{k,i}$ (the uncoupled condition) and the coupled sign pattern
$(+1,\ldots,+1) \in H$ (the identity is always in $H$).
Hence $N \cap S_k^{(k,i)} = P_{k,i} = \pi_{k,i}(N)$. This gives a trivial
coupling $N_{k,i}/M_{k,i} = 1$ and $(k,i) \notin J(N)$.

\smallskip
Let \noindent\textit{$\mathbf{P}(N) = \mathbf{P}$ and $H(N) = H$.}
For $(k,i) \notin J(N) = J$, $P_{k,i}(N) = \pi_{k,i}(N) = P_{k,i}$, so
$\mathbf{P}(N) = \mathbf{P}$.
For $H(N)$: by definition
$H(N) = \{(\sgn g_{k,i})_{(k,i) \in J} : g \in N\}$.
Every $g \in N(J,\mathbf{P},H)$ satisfies $(\sgn g_{k,i})_{(k,i) \in J} \in H$,
giving $H(N) \subseteq H$; and for every $h \in H$, choosing
$g_{k',i'} \in S_{k'}$ with $\sgn g_{k',i'} = h_{(k',i')}$ for $(k',i') \in J$
and any $g_{k',i'} \in P_{k',i'}$ for $(k',i') \notin J$ gives $g \in N$
with sign-pattern $h$, giving $H \subseteq H(N)$.
Hence $H(N) = H$.

\smallskip
Combined with injectivity (from above) , $\Phi$ is a bijection from
$\mathcal{N}(G)$ onto the set of admissible triples.
\end{proof}

\begin{remark}[New element types in $\mathcal{N}(S_3^3)$]\label{rem:examples}
For $G = S_3^3$, Lemma~\ref{lem:lattice} gives $|\mathcal{N}(G)| = 38$;
the two classical families account for only $27 + 4 = 31$ elements.
The seven additional mixed elements illustrate the complete classification.
Index the three copies of $S_3$ as $(3,1),(3,2),(3,3)$.

\smallskip
We classify these seven elements according to the size of the coupling
set $J$ in their admissible triple $(J,\mathbf{P},H)$. The elements with
$|J| = 2$ arise when a sign constraint couples precisely two factors
while the third is restricted to a proper normal subgroup (Type~1),
and the unique element with $|J| = 3$ (Type~2) arises when all three
factors participate in a single joint sign condition.

\smallskip
\textbf{Type~1: $D_I \cap (S_3^2 \times N_r)$ (sign condition on two factors
with a constrained third factor).}
For $I = \{(3,1),(3,2)\}$ and $N_r \in \{A_3, \{e\}\}$, the admissible triple
$J = \{(3,1),(3,2)\}$, $P_{3,3} = N_r$, $H = H_{\{(3,1),(3,2)\}}$
yields a normal subgroup of order $18|N_r|$ (index~$12$ or~$4$).
Both projections $\pi_{3,1}, \pi_{3,2}$ are surjective onto $S_3$,
while $\pi_{3,3}(N) = N_r \subsetneq S_3$. The factors $(3,1)$ and $(3,2)$
are sign-correlated, while $(3,3)$ is simply restricted to $N_r$.
Varying $I$ over the three pairs in $\{(3,1),(3,2),(3,3)\}$ and $N_r$ over
$\{A_3, \{e\}\}$ produces $3 \times 2 = 6$ mixed elements
(of orders~$54$ and~$18$, respectively).

\smallskip
\textbf{Type~2: all-same-sign element (joint sign condition on all factors).}
The admissible triple $J = \{(3,1),(3,2),(3,3)\}$, $\mathbf{P} = \emptyset$,
$H = \{(+1,+1,+1),(-1,-1,-1)\}$ yields
\[
  E := \{(g_{3,1},g_{3,2},g_{3,3}) \in S_3^3 :
         \sgn g_{3,1} = \sgn g_{3,2} = \sgn g_{3,3}\},
\]
a normal subgroup of order~$54$.
All three factors are sign-active with $J(E) = \{(3,1),(3,2),(3,3)\}$.
The element $E$ coincides with the lattice meet
$D_{\{(3,1),(3,2)\}} \wedge D_{\{(3,2),(3,3)\}}$.
Together with the six elements of Type~1, this accounts for all
$7$ mixed elements and combined with $27$ sub-products and $4$ sign-parity
elements, the total $38 = |\mathcal{N}(S_3^3)|$ is exact
(verified using Sage~\cite{Sage2024} and GAP~\cite{GAP2022}).
\end{remark}

\begin{lemma}[Inclusion Characterization]\label{lem:incl-char}
For $N = N(J,\mathbf{P},H) \in \mathcal{N}(G)$ and each pair $(k,i)$ with $a_k > 0$,
define the \emph{effective component}
\[
  \widetilde{P}_{k,i}(N) \;:=\;
  \begin{cases}
    P_{k,i} & (k,i) \notin J, \\
    S_k & (k,i) \in J,
  \end{cases}
  \qquad \widetilde{P}_{k,i}(N) \in \mathcal{N}(S_k).
\]
For $P \in \mathcal{N}(S_k)$, write $\sgn(P) := \{\sgn g : g \in P\} \subseteq \{\pm 1\}$;
explicitly,
\[
  \sgn(P) \;=\; \{+1\} \quad\text{if } P \neq S_k\ (\text{equivalently } P \leq A_k),
  \qquad \sgn(P) \;=\; \{\pm 1\} \quad\text{if } P = S_k.
\]
Let $N_s = N(J_s, \mathbf{P}_s, H_s)$ for $s = 1,2$, with elements of $H_s$
written as tuples $h_s = (h_{s,(k,i)})_{(k,i) \in J_s} \in \{\pm 1\}^{J_s}$.
Then $N_1 \leq N_2$ in $\mathcal{N}(G)$ if and only if both:
\begin{enumerate}[(a)]
  \item\label{ic:a}
    $\widetilde{P}_{k,i}(N_1) \leq \widetilde{P}_{k,i}(N_2)$ in $\mathcal{N}(S_k)$
    for every $(k,i)$ with $a_k > 0$;
  \item\label{ic:b}
    for every $h_1 \in H_1$ and every choice
    $\boldsymbol{\epsilon} = (\epsilon_{k,i})_{(k,i) \in J_2 \setminus J_1}$ with
    $\epsilon_{k,i} \in \sgn(\widetilde{P}_{k,i}(N_1))$, the combined sign pattern
    $h^* \in \{\pm 1\}^{J_2}$ defined by
    \[
      h^*_{(k,i)} \;:=\;
      \begin{cases} h_{1,(k,i)} & (k,i) \in J_1 \cap J_2, \\
                    \epsilon_{k,i} & (k,i) \in J_2 \setminus J_1, \end{cases}
    \]
    satisfies $h^* \in H_2$.
\end{enumerate}
\end{lemma}

\begin{proof}
Assume $N_1 \subseteq N_2$.

\emph{(a):} The projection $\pi_{k,i}\colon G \to S_k$ is monotone, so
$\pi_{k,i}(N_1) \leq \pi_{k,i}(N_2)$ in $\mathcal{N}(S_k)$. By
Lemma~\ref{lem:lattice}, $\pi_{k,i}(N_s) = \widetilde{P}_{k,i}(N_s)$ in either case
($(k,i) \in J_s$ gives $\pi_{k,i}(N_s) = S_k$, and $(k,i) \notin J_s$ gives
$\pi_{k,i}(N_s) = P_{s,k,i}$).

\emph{(b):} Fix $h_1 \in H_1$ and $\boldsymbol{\epsilon}$ as above.
Construct $g \in N_1$ by choosing:
for $(k,i) \in J_1$, any $g_{k,i} \in S_k$ with $\sgn g_{k,i} = h_{1,(k,i)}$
(possible since $\sgn$ is surjective on $S_k$);
for $(k,i) \in J_2 \setminus J_1$, any $g_{k,i} \in P_{1,k,i}$ with
$\sgn g_{k,i} = \epsilon_{k,i}$
(possible since $\epsilon_{k,i} \in \sgn(P_{1,k,i})$);
for $(k,i) \notin J_1 \cup J_2$, any $g_{k,i} \in P_{1,k,i}$.
Then $g_{k,i} \in P_{1,k,i}$ for all $(k,i) \notin J_1$ and
$(\sgn g_{k,i})_{(k,i) \in J_1} = h_1 \in H_1$, so $g \in N_1$.
Since $N_1 \subseteq N_2$, $g \in N_2$, and the $J_2$-pattern of $g$ is
exactly $h^*$, hence $h^* \in H_2$.

Conversely assume \eqref{ic:a} and \eqref{ic:b}. Let $g \in N_1$,
we show that $g \in N_2$.

\emph{Per-coordinate condition for $(k,i) \notin J_2$.}
If $(k,i) \notin J_1$, then
$g_{k,i} \in P_{1,k,i} = \widetilde{P}_{k,i}(N_1) \leq \widetilde{P}_{k,i}(N_2) = P_{2,k,i}$
by~\eqref{ic:a}. If $(k,i) \in J_1$, then~\eqref{ic:a} gives
$S_k = \widetilde{P}_{k,i}(N_1) \leq \widetilde{P}_{k,i}(N_2) = P_{2,k,i}$, forcing
$P_{2,k,i} = S_k$ and so $g_{k,i} \in P_{2,k,i}$ trivially.

\emph{Sign-pattern condition on $J_2$.} Define
$\eta := (\sgn g_{k,i})_{(k,i) \in J_1}$ and
$\boldsymbol{\zeta} := (\sgn g_{k,i})_{(k,i) \in J_2 \setminus J_1}$.
Then $\eta \in H_1$ since $g \in N_1$, and
$\zeta_{k,i} \in \sgn(P_{1,k,i}) = \sgn(\widetilde{P}_{k,i}(N_1))$ since
$g_{k,i} \in P_{1,k,i}$. Apply~\eqref{ic:b} with $h_1 = \eta$ and
$\boldsymbol{\epsilon} = \boldsymbol{\zeta}$: the resulting combined pattern
is exactly $(\sgn g_{k,i})_{(k,i) \in J_2}$, which therefore lies in $H_2$.
Hence $g \in N_2$.
\end{proof}

\begin{remark}\label{rem:incl-char-pi}
The identity $\pi_{k,i}(N) = \widetilde{P}_{k,i}(N)$ used in step~\emph{(a)} is part
of the content of Lemma~\ref{lem:lattice} (Step~2 of its proof): for
$(k,i) \notin J(N)$ the trivial coupling forces $\pi_{k,i}(N) = P_{k,i}(N)$, and for
$(k,i) \in J(N)$ the $C_2$-coupling forces $\pi_{k,i}(N) = S_k$.
\end{remark}

\begin{remark}[Exhaustive description of $\mathcal{N}(G)$ and meet/join
formulas]\label{rem:goursat-structure}
The bijection argument of Lemma~\ref{lem:lattice} partitions $\mathcal{N}(G)$
into three pairwise disjoint families that together exhaust the lattice.
We record their explicit descriptions, cardinalities, and the lattice
operations between sign-parity elements, as these are used throughout
the arguments that follow.

\medskip
\noindent\textbf{(1) Sub-products} correspond to triples with $J = \emptyset$.
Each is a direct product $\prod_{(k,i)} P_{k,i}^{(k,i)}$ with
$P_{k,i} \in \mathcal{N}(S_k)$ chosen independently. This gives
$\prod_{k:\,a_k>0} |\mathcal{N}(S_k)|^{a_k}$ elements in total.

\medskip
\noindent\textbf{(2) Sign-parity elements} correspond to triples
$(I, (S_k)_{(k,i) \notin I}, H_I)$ with $J = I \neq \emptyset$,
where $H_I = \{h \in \{\pm 1\}^I : \prod_{(k,i) \in I} h_{(k,i)} = 1\}$
is the product-one kernel.
Admissibility forces $|I| \geq 2$. If $|I| = 1$, condition~\eqref{adm:i}
requires $(-1_{(k,i)}) \notin H$ while condition~\eqref{adm:ii} requires
some element of $H$ to have $-1$ in coordinate $(k,i)$, this gives a contradiction.
There are accordingly $2^T - T - 1$ sign-parity elements $D_I$,
where $T = \sum_{k:\,a_k>0} a_k$ is the total number of copy-slots,
indexed by subsets $I$ of the full index set $\{(k,i)\}$ with $|I| \geq 2$.

\medskip
\noindent\textbf{(3) Mixed elements} comprise all remaining admissible
triples, those with $J \neq \emptyset$ that fail to be sign-parity
elements, meaning either $H \neq H_J$ or $P_{k,i} \neq S_k$ for
some $(k,i) \notin J$.
Every mixed element is expressible as a finite lattice meet of
sub-products and sign-parity elements. To see this, let
$N = N(J,\mathbf{P},H)$ be mixed and define the sub-product
\[
  S_{\mathbf{P}} \;:=\; \prod_{(k,i) \notin J} P_{k,i}^{(k,i)} \;\times\;
  \prod_{(k,i) \in J} S_k^{(k,i)}.
\]
Identifying $\{\pm1\}^J$ with $\mathbb{F}_2^J$ via the map 
$\sgn(g) \mapsto \tfrac{1-\sgn(g)}{2}$, the subgroup $H$ becomes
a $\mathbb{F}_2$-subspace. By admissibility~\eqref{adm:ii},
for each $(k,i) \in J$, there exists some $h \in H$ has $h_{(k,i)} = -1$,
so the unit functional $e_{(k,i)}^*$ does not vanish on $H$ and hence
$e_{(k,i)}^* \notin H^\perp$.
Therefore, every nonzero element of $H^\perp \leq (\mathbb{F}_2^J)^*$
has support of size at least $2$.
Fix a $\mathbb{F}_2$-basis $\{f_1,\ldots,f_\ell\}$ of $H^\perp$ such that 
each $f_i$ has support $I_i \subseteq J$ with $|I_i| \geq 2$; this implies
$D_{I_i}$ is a valid sign-parity element.
The double-annihilator identity $(H^\perp)^\perp = H$ then gives
$H = \bigcap_{i=1}^\ell \ker(f_i)$, and therefore
\[
  N(J,\mathbf{P},H) \;=\; S_{\mathbf{P}} \;\wedge\; D_{I_1} \;\wedge\; \cdots \;\wedge\; D_{I_\ell}.
\]

\medskip
\noindent\textbf{Lattice operations on sign-parity elements.}
The sign-parity elements form an antichain in $\mathcal{N}(G)$, i.e.,
two distinct elements $D_I$ and $D_J$ are incomparable, since each
has index~$2$ in $G$ and neither contains the other as a proper subgroup.
For distinct $D_I, D_J$ with $|I|, |J| \geq 2$, the meet and join are:
\begin{equation}\label{eq:meet}
  D_I \wedge D_J \;=\; D_I \cap D_J
  \;=\; \Bigl\{g \in G :
  {\textstyle\prod_{(k,i) \in I}}\,\sgn g_{k,i} = 1
  \;\text{ and }\;
  {\textstyle\prod_{(k,i) \in J}}\,\sgn g_{k,i} = 1
  \Bigr\},
\end{equation}
a normal subgroup of index~$4$ in $G$. In the canonical parameterisation,
$D_I \wedge D_J$ has coupling set $J_{\mathrm{meet}} = I \cup J$,
local data $P_{k,i} = S_k$ for $(k,i) \notin I \cup J$, and sign subgroup
\[
  H_{\mathrm{meet}} = \Bigl\{h \in \{\pm1\}^{I \cup J} :
  {\textstyle\prod_{(k,i) \in I}} h_{(k,i)} = 1
  \;\text{ and }\;
  {\textstyle\prod_{(k,i) \in J}} h_{(k,i)} = 1\Bigr\}.
\]
This element is mixed whenever $I$ and $J$ are incomparable under
inclusion, and equals $D_J$ (resp.\ $D_I$) when $I \subseteq J$
(resp.\ $J \subseteq I$).
For the join, since $D_I$ and $D_J$ are distinct subgroups of index~$2$,
their product is all of $G$:
\[
  D_I \vee D_J = G \qquad \text{for } I \neq J.
\]
\end{remark}
\section{The Factor Permutation Lemma}
\label{sec:injection}
We have established a complete combinatorial description of $\mathcal{N}(G)$,
we turn to the structure of $\LatAut(G)$ itself. The central observation is that the complemented elements of
$\mathcal{N}(G)$ are exactly the full sub-products of the individual
factors. Among these, the nontrivial indecomposable complemented
elements are precisely the direct factors $S_k^{(k,i)}$, and this
characterisation is preserved by lattice automorphisms.
It follows that every $\varphi \in \LatAut(G)$
induces a permutation of $\{S_k^{(k,i)}\}$. The chain length
of the local interval $[1, S_k^{(k,i)}]$, equal to $3$ when
$k = 4$ and $2$ otherwise, is a lattice invariant, so this
permutation must moreover respect the bipartition
$\mathcal{A} \sqcup \mathcal{B}$ of all copy-slots.
Composing these observations gives an injective homomorphism
$\LatAut(G) \hookrightarrow S_{a_4} \times S_B$.

\begin{lemma}[Factor Permutation]\label{lem:factperm}
Let $G = \prod_{k \geq 3} S_k^{a_k}$ be a tower group, with individual
factors $S_k^{(k,i)}$ for $k \geq 3$, $a_k > 0$, $1 \leq i \leq a_k$.
Every $\varphi \in \LatAut(G)$ permutes the set of individual factors
$\{S_k^{(k,i)}\}$. Moreover, $\varphi$ preserves the partition
\[
  \mathcal{A} = \{(4,i) : 1 \leq i \leq a_4\}, \qquad
  \mathcal{B} = \{(k,i) : k \geq 3,\; k \neq 4,\; 1 \leq i \leq a_k\},
\]
so the induced permutation lies in $\operatorname{Sym}(\mathcal{A}) \times \operatorname{Sym}(\mathcal{B})
\cong S_{a_4} \times S_B$. This gives an injective homomorphism
\[
  \LatAut(G) \hookrightarrow S_{a_4} \times S_B.
\]
\end{lemma}

\begin{proof}
Each individual factor $S_k^{(k,i)}$ is an element of $\mathcal{N}(G)$, and we
claim it is a \emph{lattice-theoretic direct factor}. There is a unique complement
$C_{k,i} := \prod_{(k',i') \neq (k,i)} S_{k'}^{(k',i')} \in \mathcal{N}(G)$
satisfying $S_k^{(k,i)} \wedge C_{k,i} = 1$ and
$S_k^{(k,i)} \vee C_{k,i} = G$. Any complement $C$ must satisfy
$C \cap S_k^{(k,i)} = 1$ and $C \cdot S_k^{(k,i)} = G$,
giving $G = S_k^{(k,i)} \times C$ as an internal direct product
of normal subgroups. By Theorem~\ref{thm:remak}, since $Z(G) = 1$
(each $S_k$ has trivial center for $k \geq 3$), the indecomposable
factors of any two direct decompositions of $G$ coincide as subgroups,
not merely up to isomorphism. Since $S_k^{(k,i)}$ is group-indecomposable
(its normal subgroup lattice $\mathcal{N}(S_k)$ is a chain, so $S_k$ admits
no non-trivial direct decomposition),
the complement $C_{k,i}$ is the unique element of $\mathcal{N}(G)$ with
the required meet and join, and $S_k^{(k,i)}$ is recoverable from $\mathcal{N}(G)$
as the element with this unique complement.

No element of $\mathcal{N}(G)$ other than a sub-product has a complement
in $\mathcal{N}(G)$. If $N \cap C = 1$ and $N \vee C = G$ for some
$C \in \mathcal{N}(G)$, then $G = N \times C$ as an internal direct product
of normal subgroups. Since conjugation by any $(n,c) \in N \times C$ acts
on $N$ via conjugation by $n$ alone, every indecomposable direct summand
of $N$ is normal in $G$. The resulting decomposition of $G$ into indecomposable
normal subgroups is then another Krull--Schmidt decomposition of $G$. By
Theorems~\ref{thm:ks} and~\ref{thm:remak}, its factors are permuted images of
$\{S_k^{(k,i)}\}$ under inner automorphisms of $G$. Since each $S_k^{(k,i)}$
is normal in $G$, every inner automorphism fixes it setwise, so the
indecomposable summands of $N$ are exactly $\{S_k^{(k,i)} : (k,i) \in T\}$
for some $T \subseteq \mathcal{A} \cup \mathcal{B}$.
Hence $N = \prod_{(k,i) \in T} S_k^{(k,i)}$, a sub-product element.
Since sign-parity and mixed elements both have $J \neq \emptyset$ in their
canonical triple (Lemma~\ref{lem:lattice}) and neither of them is a sub-product,
neither admits a complement in $\mathcal{N}(G)$.

Since lattice isomorphisms preserve meets, joins, and unique-complement pairs,
every $\varphi \in \LatAut(G)$ sends each $S_k^{(k,i)}$ to some other factor
$S_{k'}^{(k',i')}$. This defines a permutation $\pi_\varphi$ on $\mathcal{A} \cup \mathcal{B}$.
The permutation respects the partition $\{\mathcal{A}, \mathcal{B}\}$ because the
interval $[1, S_k^{(k,i)}]$ in $\mathcal{N}(G)$ is isomorphic to $\mathcal{N}(S_k)$,
which is a chain of length $3$ when $k = 4$ (namely $\{1, V_4, A_4, S_4\}$)
and a chain of length $2$ when $k \geq 3$, $k \neq 4$ (namely
$\{1, A_k, S_k\}$). Chains of distinct lengths are non-isomorphic, so
$\varphi$ sends $\mathcal{A}$-factors to $\mathcal{A}$-factors and $\mathcal{B}$-factors
to $\mathcal{B}$-factors. The induced permutation lies in
$\operatorname{Sym}(\mathcal{A}) \times \operatorname{Sym}(\mathcal{B}) \cong S_{a_4} \times S_B$.

The map $\varphi \mapsto \pi_\varphi$ is a homomorphism since
$\pi_{\varphi \circ \psi} = \pi_\varphi \circ \pi_\psi$ by construction.
It is injective. Suppose $\pi_\varphi = \id$, so $\varphi$ fixes every
factor $S_k^{(k,i)}$.

\medskip

Each interval $[1, S_k^{(k,i)}]$ is a chain, so its only lattice
automorphism is the identity. Hence $\varphi$ fixes every element of
$[1, S_k^{(k,i)}]$, in particular $A_k^{(k,i)}$ (and $V_4^{(4,i)}$
when $k = 4$). Since $\varphi$ fixes every factor $S_k^{(k,i)}$
and preserves joins, it also fixes
$C_{k,i} = \bigvee_{(k',i') \neq (k,i)} S_{k'}^{(k',i')}$.
Every sub-product $\prod_{(k,i)} P_{k,i}^{(k,i)}$ equals the join
of elements from these fixed intervals. Since $\varphi$ preserves joins,
it fixes every sub-product element.

\medskip

For any $N \in \mathcal{N}(G)$, the pair $(k,i)$ lies in $J(N)$ if and only
if $N \wedge S_k^{(k,i)} = A_k^{(k,i)}$ and
$N \vee C_{k,i} = G$.
Since $\varphi$ preserves meets and joins and fixes
$S_k^{(k,i)}$, $A_k^{(k,i)}$, and $C_{k,i}$, it follows
that $J(\varphi(N)) = J(N)$ for every $N$.

The coatoms of $\mathcal{N}(G)$ are precisely the index-$2$ normal subgroups
of $G$. Every surjective homomorphism $G \to \{\pm 1\}$ has the form
$g \mapsto \prod_{(k,i) \in I} \sgn g_{k,i}$ for some non-empty
$I \subseteq \mathcal{A} \cup \mathcal{B}$: since each $S_k$ ($k \geq 3$) has
$A_k$ as its unique index-$2$ subgroup
(Lemma~\ref{lem:lattice}), the only surjective homomorphism
$S_k \to \{\pm 1\}$ is $\sgn$, so any surjective homomorphism
on $G$ is a product of a non-empty selection of these.
The kernel is a sub-product coatom when $|I| = 1$ and a sign-parity element
$D_I$ when $|I| \geq 2$. Hence every coatom of $\mathcal{N}(G)$ is either
a sub-product coatom or a sign-parity element.

Each sign-parity coatom $D_I$ maps under $\varphi$ to a coatom of
$\mathcal{N}(G)$; since $\varphi$ fixes all sub-product coatoms,
$\varphi(D_I) = D_{I'}$ for some sign-parity $D_{I'}$.
To identify $I'$, we verify $J(D_{I'}) = I'$ using both conditions in
Definition~\ref{def:triples}: an element with $g_{k,i}$ in position $(k,i)$
and identity elsewhere lies in $D_{I'}$ if and only if $\sgn(g_{k,i}) = 1$
when $(k,i) \in I'$ (and freely otherwise), giving
$D_{I'} \wedge S_k^{(k,i)} = A_k^{(k,i)}$ if and only if $(k,i) \in I'$;
and $\pi_{k,i}(D_{I'}) = S_k$ for every $(k,i) \in I'$, since $|I'| \geq 2$
permits choosing the sign of one other factor in $I'$ to satisfy the
product constraint for any target $g_{k,i} \in S_k$.
Hence $J(D_{I'}) = I'$, combined with $J(\varphi(D_I)) = J(D_I) = I$,
this gives $I' = I$, so $\varphi(D_I) = D_I$.

\medskip

By Remark~\ref{rem:goursat-structure}, every mixed element is a meet of
sub-product and sign-parity elements. Since $\varphi$ preserves meets
and fixes every sub-product and every sign-parity element, 
it fixes every mixed element. Hence $\varphi = \id$.
\end{proof}

\section{The Realisation Lemma}
\label{sec:realisation}

The injection $\pi \colon \LatAut(G) \hookrightarrow S_{a_4} \times S_B$
of Lemma~\ref{lem:factperm} is an isomorphism if and only if every
element of $\operatorname{Sym}(\mathcal{A}) \times \operatorname{Sym}(\mathcal{B})$
is realised as $\pi_\varphi$ for some $\varphi \in \LatAut(G)$.
We show that for each $\sigma \in \operatorname{Sym}(\mathcal{A}) \times
\operatorname{Sym}(\mathcal{B})$, an explicit $\tau_\sigma \in \LatAut(G)$
with $\pi_{\tau_\sigma} = \sigma$, constructed directly on admissible
triples via the unique chain isomorphism
$\psi_{(k,i),(k',i')} \colon \mathcal{N}(S_k) \to \mathcal{N}(S_{k'})$
between equal-length chains, where $(k',i') := \sigma(k,i)$ denotes the
image of the copy-slot $(k,i)$ under $\sigma$. (The chain isomorphism
depends only on the source type $k$ and target type $k'$, not on the
copy indices $i, i'$; these appear in the subscript only to identify
which copy-slot is being mapped.)
Order preservation is verified using
Lemma~\ref{lem:incl-char} and no case split on whether $k = k'$
is needed.

\begin{lemma}[Realisation]\label{lem:realise}
Let $\sigma = (\sigma_\mathcal{A}, \sigma_\mathcal{B})$ be any element of
$\operatorname{Sym}(\mathcal{A}) \times \operatorname{Sym}(\mathcal{B})$.
Then there exists $\tau_\sigma \in \LatAut(G)$ with $\pi_{\tau_\sigma} = \sigma$.
Hence the map from Lemma~\ref{lem:factperm} is an isomorphism, and thus
\[
  \LatAut(G) \;\cong\; S_{a_4} \times S_B.
\]
\end{lemma}

\begin{proof}
Since $\sigma$ preserves the partition $\{\mathcal{A}, \mathcal{B}\}$, for
every $(k,i)$ the chains $\mathcal{N}(S_k)$ and $\mathcal{N}(S_{k'})$
(where $\sigma(k,i) = (k',i')$)
have the same length: $3$ for $(k,i) \in \mathcal{A}$, and $2$ for
$(k,i) \in \mathcal{B}$. Let
\[
  \psi_{(k,i),\,(k',i')} \colon \mathcal{N}(S_k)
  \xrightarrow{\;\sim\;} \mathcal{N}(S_{k'})
\]
denote the unique order-preserving bijection between these chains,
explicitly, $1 \mapsto 1$, $A_k \mapsto A_{k'}$, and
(when $k = k' = 4$) $V_4 \mapsto V_4$, $S_4 \mapsto S_4$.
These maps compose:
$\psi_{(k',i'),(k'',i'')} \circ \psi_{(k,i),(k',i')} = \psi_{(k,i),(k'',i'')}$
by uniqueness of the chain isomorphism between equal-length finite chains.

\medskip
\noindent\textit{Construction.}
For each $N = N(J, \mathbf{P}, H) \in \mathcal{N}(G)$, define
\begin{equation}\label{eq:tau}
  \tau_\sigma(N) \;:=\; N\!\bigl(\sigma(J),\; \mathbf{P}^\sigma,\; H^\sigma\bigr),
\end{equation}
where, writing $\sigma(k,i) = (k',i')$, for $(k',i') \notin \sigma(J)$
(equivalently $(k,i) = \sigma^{-1}(k',i') \notin J$),
\[
  P^\sigma_{k',i'} \;:=\; \psi_{(k,i),\,(k',i')}\!\bigl(P_{k,i}\bigr)
  \;\in\; \mathcal{N}(S_{k'}),
\]
and
\[
  H^\sigma \;:=\; \bigl\{\,(h_{\sigma^{-1}(k',i')})_{(k',i') \in \sigma(J)}
  \;:\; h \in H\,\bigr\} \;\leq\; \{\pm 1\}^{\sigma(J)}.
\]
(For $(k',i') \in \sigma(J)$, $\sigma^{-1}(k',i') \in J$, so the coordinate
$h_{\sigma^{-1}(k',i')}$ of $h$ is defined.)

\medskip
\noindent\textit{Output is admissible.}
Conditions~\eqref{adm:i} and~\eqref{adm:ii} on $H^\sigma$ at coordinate
$(k',i') \in \sigma(J)$ correspond, via the relabelling $\sigma$, to the same
conditions on $H$ at coordinate $(k,i) = \sigma^{-1}(k',i') \in J$, which hold by
admissibility of $(J, \mathbf{P}, H)$. The relabelling map
$h \mapsto (h_{\sigma^{-1}(k',i')})_{(k',i')}$ is multiplicative, so $H^\sigma$
inherits the subgroup structure of $H$ in $\{\pm 1\}^{\sigma(J)}$.

\medskip
\noindent\textit{Sign-content invariance.}
For $P \in \mathcal{N}(S_k)$ with $\sigma(k,i) = (k',i')$,
\begin{equation}\label{eq:sgn-invariant}
  \sgn\bigl(\psi_{(k,i),\,(k',i')}(P)\bigr) \;=\; \sgn(P).
\end{equation}
Here, $\sgn(P) = \{+1\}$ iff $P$ is not the top of $\mathcal{N}(S_k)$;
since $\psi_{(k,i),(k',i')}$ preserves chain position, this is equivalent to
$\psi_{(k,i),(k',i')}(P)$ not being the top of $\mathcal{N}(S_{k'})$,
which is in turn equivalent to $\sgn(\psi_{(k,i),(k',i')}(P)) = \{+1\}$.

\medskip
\noindent\textit{Effective-component transformation.}
For every $(k',i')$, writing $(k,i) = \sigma^{-1}(k',i')$,
\begin{equation}\label{eq:Ptilde-trans}
  \widetilde{P}_{k',i'}(\tau_\sigma N) \;=\;
  \psi_{(k,i),\,(k',i')}\!\bigl(\widetilde{P}_{k,i}(N)\bigr).
\end{equation}
If $(k',i') \in \sigma(J)$, then $(k,i) \in J$ and both sides equal
$S_{k'}$ (the chain isomorphism sends top to top). If $(k',i') \notin \sigma(J)$,
then $(k,i) \notin J$ and both sides equal
$\psi_{(k,i),(k',i')}(P_{k,i}) = P^\sigma_{k',i'}$.

\medskip
\noindent\textit{Order preservation.}
We show $N_1 \leq N_2$ in $\mathcal{N}(G)$ if and only if
$\tau_\sigma(N_1) \leq \tau_\sigma(N_2)$, using Lemma~\ref{lem:incl-char}.

For condition~\eqref{ic:a}: by~\eqref{eq:Ptilde-trans},
$\widetilde{P}_{k',i'}(\tau_\sigma N_1) \leq \widetilde{P}_{k',i'}(\tau_\sigma N_2)$ in
$\mathcal{N}(S_{k'})$ if and only if
$\psi_{(k,i),(k',i')}(\widetilde{P}_{k,i}(N_1)) \leq
\psi_{(k,i),(k',i')}(\widetilde{P}_{k,i}(N_2))$, which (since
$\psi_{(k,i),(k',i')}$ is order-preserving) is equivalent to
$\widetilde{P}_{k,i}(N_1) \leq \widetilde{P}_{k,i}(N_2)$.
As $(k',i')$ ranges over all copy-slots, so does $(k,i) = \sigma^{-1}(k',i')$,
so the condition for the LHS at every $(k',i')$ is equivalent to the condition
for the RHS at every $(k,i)$.

For condition~\eqref{ic:b}: each $h_1 \in H_1$ corresponds bijectively to
$h_1^\sigma \in H_1^\sigma$ with $(h_1^\sigma)_{(k',i')} = h_{1,\sigma^{-1}(k',i')}$ for
$(k',i') \in \sigma(J_1)$, and similarly for $H_2$. The free coordinate set
$J_2 \setminus J_1$ corresponds under $\sigma$ to
$\sigma(J_2) \setminus \sigma(J_1)$. By~\eqref{eq:sgn-invariant} and
\eqref{eq:Ptilde-trans},
$\sgn(\widetilde{P}_{k,i}(N_1)) = \sgn(\widetilde{P}_{\sigma(k,i)}(\tau_\sigma N_1))$
for every $(k,i)$, so the allowed sign choices $\boldsymbol{\epsilon}$
correspond bijectively. The combined pattern $h^* \in \{\pm 1\}^{J_2}$
maps under relabelling to $h^{*\sigma} \in \{\pm 1\}^{\sigma(J_2)}$, with
$h^* \in H_2$ if and only if $h^{*\sigma} \in H_2^\sigma$.
Hence~\eqref{ic:b} on the LHS is equivalent to~\eqref{ic:b} on the RHS.

By Lemma~\ref{lem:incl-char}, $\tau_\sigma$ preserves $\leq$ in both
directions; combined with bijectivity (immediate from
$\tau_\sigma \circ \tau_{\sigma^{-1}} = \tau_{\id} = \id$, established
below), $\tau_\sigma \in \LatAut(G)$.

\medskip
\noindent\textit{Homomorphism property.}
We verify $\tau_\sigma \circ \tau_\rho = \tau_{\sigma\rho}$ for any
$\sigma, \rho \in \operatorname{Sym}(\mathcal{A}) \times \operatorname{Sym}(\mathcal{B})$.
For $(k'',i'') \notin \sigma\rho(J)$, write $(k',i') = \rho^{-1}(k'',i'')$
and $(k,i) = \sigma^{-1}(k',i') = (\sigma\rho)^{-1}(k'',i'')$. Then
\begin{align*}
  (P^\rho)^\sigma_{k'',i''}
  &= \psi_{(k',i'),\,(k'',i'')}\!\bigl((P^\rho)_{k',i'}\bigr) \\
  &= \psi_{(k',i'),\,(k'',i'')}\!\bigl(\psi_{(k,i),\,(k',i')}(P_{k,i})\bigr) \\
  &= \psi_{(k,i),\,(k'',i'')}(P_{k,i})
  \;=\; P^{\sigma\rho}_{k'',i''},
\end{align*}
using chain-iso composition $\psi_{(k',i'),(k'',i'')} \circ \psi_{(k,i),(k',i')}
= \psi_{(k,i),(k'',i'')}$.
For each $h \in H$, the corresponding element of $(H^\rho)^\sigma$ has
$(k'',i'')$-th coordinate $h_{\rho^{-1}\sigma^{-1}(k'',i'')} = h_{(\sigma\rho)^{-1}(k'',i'')}$,
which equals the $(k'',i'')$-th coordinate of the corresponding element of
$H^{\sigma\rho}$. Hence $\tau_\sigma \circ \tau_\rho = \tau_{\sigma\rho}$.
Taking $\rho = \sigma^{-1}$ gives $\tau_\sigma \circ \tau_{\sigma^{-1}} = \id$,
so each $\tau_\sigma$ is a bijection of $\mathcal{N}(G)$.

\medskip
\noindent\textit{Realisation of $\sigma$.}
The factor $S_k^{(k,i)}$ corresponds under Lemma~\ref{lem:lattice} to the
admissible triple $(\emptyset,\,(P_{k',i'})_{k',i'},\,\{()\})$ with $P_{k,i} = S_k$
and $P_{k',i'} = 1$ for $(k',i') \neq (k,i)$. Apply~\eqref{eq:tau} with
$J = \emptyset$ (so $\sigma(J) = \emptyset$ and $H^\sigma = \emptyset$):
write $\sigma(k,i) = (k'',i'')$. Then
\[
  P^\sigma_{k'',i''} = \psi_{(k,i),(k'',i'')}(S_k) = S_{k''}
  \quad\text{and}\quad
  P^\sigma_{k',i'} = \psi_{\sigma^{-1}(k',i'),(k',i')}(1) = 1
  \;\text{ for } (k',i') \neq (k'',i''),
\]
since chain isomorphisms send the bottom element $1$ to $1$.
Hence $\tau_\sigma(S_k^{(k,i)}) = S_{k''}^{(k'',i'')}$, so
$\pi_{\tau_\sigma} = \sigma$.

Combined with the injection $\LatAut(G) \hookrightarrow S_{a_4} \times S_B$
of Lemma~\ref{lem:factperm}, this gives $\LatAut(G) \cong S_{a_4} \times S_B$.
\end{proof}

\begin{remark}\label{rem:realise-vs-group}
For a transposition $\sigma$ swapping $(k,i) \leftrightarrow (k,i')$ with
the same type $k$, the lattice automorphism $\tau_\sigma$ is induced by
the group automorphism of $G$ that exchanges the $i$-th and $i'$-th copies
of $S_k$ via any group isomorphism $S_k \to S_k$ (the choice here is
immaterial since all such isomorphisms induce the same chain map
$\psi_{(k,i),(k,i')}$, which is the identity on $\mathcal{N}(S_k)$).
For transpositions swapping $(k,i) \leftrightarrow (k',i')$ with $k \neq k'$
(only possible within class $\mathcal{B}$), no group isomorphism
$S_k \to S_{k'}$ exists, and $\tau_\sigma$ has no group-automorphism
realisation; the construction is genuinely lattice-theoretic, relying on
Lemma~\ref{lem:incl-char} and the non-trivial chain isomorphism
$\psi_{(k,i),(k',i')}$ (e.g., $A_k \mapsto A_{k'}$).
\end{remark}

\section{The Product Formula and Termination}

\label{sec:termination}

Lemmas~\ref{lem:factperm} and~\ref{lem:realise} together prove
Theorem~\ref{thm:product}. We now derive Theorem~\ref{thm:termination}.

To compute $\LatAut$ of the groups that appear in the tower, we need two
additional computations.

\begin{lemma}\label{lem:single}
For $n \geq 3$, $\LatAut(S_n) = 1$.
\end{lemma}

\begin{proof}
$\mathcal{N}(S_n)$ is a finite chain (of length $2$ for $n \neq 4$ and length $3$
for $n = 4$). Any automorphism of a chain must fix every element (since it fixes
the unique minimum and maximum and the chain is totally ordered). Hence the only
lattice automorphism of $\mathcal{N}(S_n)$ is the identity.
\end{proof}

\begin{lemma}\label{lem:c2sq}
$\LatAut(C_2^2) \cong \GL(2, 2) \cong S_3$.
\end{lemma}

\begin{proof}
The normal subgroups of $C_2^2$ are $\{1\}$, three subgroups of order $2$
(one for each nonzero element), and $C_2^2$ itself. A lattice automorphism
must fix $\{1\}$ and $C_2^2$ (the unique minimum and maximum) and is
therefore determined by its restriction to the three atoms. This gives an
injective homomorphism $\LatAut(C_2^2) \hookrightarrow S_3$.

For the reverse inclusion, identify $C_2^2$ with $\mathbb{F}_2^2$ as an
additive group. Each nonzero vector spans a unique $1$-dimensional subspace
(since $\mathbb{F}_2^* = \{1\}$), so the three atoms of $\mathcal{N}(C_2^2)$
correspond bijectively to the three nonzero elements of $\mathbb{F}_2^2$.
Every $A \in \GL(2, \mathbb{F}_2)$ permutes these nonzero elements and maps
subgroups to subgroups, hence induces a lattice automorphism of
$\mathcal{N}(C_2^2)$. The induced map $\GL(2, \mathbb{F}_2) \to \LatAut(C_2^2)$
is injective: a linear map fixing all nonzero vectors of $\mathbb{F}_2^2$
is the identity. Since $|\GL(2, \mathbb{F}_2)| = (4-1)(4-2) = 6 = |S_3|$,
the embedding $\LatAut(C_2^2) \hookrightarrow S_3$ is an isomorphism, and
$\GL(2, \mathbb{F}_2) \cong S_3$.
\end{proof}

\begin{lemma}\label{lem:c2sm}
For $m \geq 3$, $\LatAut(C_2 \times S_m) \cong C_2$.
\end{lemma}

\begin{proof}
Let  $G = C_2 \times S_m$ with $C_2 = \{1, c\}$. The elements of $\mathcal{N}(G)$
are:
\begin{itemize}
\item Sub-products: $1$, $\{1\} \times A_m$, $C_2 \times \{1\}$, $\{1\} \times S_m$,
  $C_2 \times A_m$, and $G$. (For $m = 4$, also $\{1\} \times V_4$ and $C_2 \times V_4$,
  but the argument below applies uniformly.)
\item One sign-parity element $D = \{(e, g) : \sgn(g) = 1\} \cup \{(c, g) : \sgn(g) = -1\}$,
  corresponding to the constraint $\chi_1 \cdot \chi_2 = 1$ where $\chi_1$ is the
  non-trivial character of $C_2$ and $\chi_2 = \sgn$.
\end{itemize}

We use the lattice structure to fix $C_2 \times \{1\}$. Counting the 
coatoms,  each atom of $\mathcal{N}(G)$ lies below:
\begin{itemize}
  \item $C_2 \times \{1\}$ lies below \emph{exactly one} coatom
    (namely $C_2 \times A_m$):
    $(c,e) \notin \{1\} \times S_m$ since first coordinate $c \neq 1$;
    $(c,e) \notin D$ since $\chi_1(c)\sgn(e) = (-1)(+1) = -1 \neq 1$;
    but $(c,e) \in C_2 \times A_m$ since $c \in C_2$ and $e \in A_m$.
  \item $\{1\} \times A_m$ (or $\{1\} \times V_4$ when $m = 4$) lies below
    \emph{all three} coatoms: it is contained in $\{1\} \times S_m$
    and in $C_2 \times A_m$ as sub-products; and every $(1, a)$ with
    $a \in A_m$ satisfies $\chi_1(1)\sgn(a) = 1$, so
    $\{1\} \times A_m \leq D$.
\end{itemize}
Since lattice automorphisms permute coatoms and preserve incidence,
the count of coatoms above each atom is a lattice invariant.
As $1 \neq 3$, no $\varphi \in \LatAut(G)$ can swap the two atoms,
so $\varphi(C_2 \times \{1\}) = C_2 \times \{1\}$.

Once $C_2 \times \{1\}$ is fixed, $\varphi$ must fix the unique element of
$\mathcal{N}(G)$ covering it. For $m \neq 4$, the minimal normal subgroup
of $S_m$ is $A_m$ (simple for $m \geq 5$; equal to $C_3$ for $m = 3$),
so $C_2 \times A_m$ covers $C_2 \times \{1\}$ with no intermediate normal
subgroup, and is itself a coatom. For $m = 4$, the minimal normal subgroup
of $S_4$ is $V_4$, and no normal subgroup of $C_2 \times S_4$ lies strictly
between $C_2 \times \{1\}$ and $C_2 \times V_4$, so $C_2 \times V_4$ covers
$C_2 \times \{1\}$ and is fixed; $\varphi$ then also fixes $C_2 \times A_4$,
the unique element of $\mathcal{N}(G)$ covering $C_2 \times V_4$, which is a
coatom since $[G : C_2 \times A_4] = 2$.
In both cases the coatom $C_2 \times A_m$ is fixed by $\varphi$.
The remaining two coatoms are $\{1\} \times S_m$ and $D$. We show they are
lattice-equivalent by comparing their intervals:
\[
  [1, \{1\} \times S_m] \;\cong\; \mathcal{N}(S_m), \qquad
  [1, D] \;\cong\; \mathcal{N}(S_m),
\]
where for $m \neq 4$ each interval is the chain
$\{1\} \subset \{1\} \times A_m \subset \{1\} \times S_m$
(resp.\ $\{1\} \subset \{1\} \times A_m \subset D$) of length~$2$,
and for $m = 4$ each interval is the chain
$\{1\} \subset \{1\} \times V_4 \subset \{1\} \times A_4 \subset \{1\} \times S_4$
(resp.\ $\{1\} \subset \{1\} \times V_4 \subset \{1\} \times A_4 \subset D$)
of length~$3$. To verify: the elements of $\mathcal{N}(G)$ below $D$ are
$\{1\}$, $\{1\} \times A_m$ (and $\{1\} \times V_4$ when $m=4$), and $D$ itself,
forming a chain isomorphic to $[1, \{1\} \times S_m]$.
(Note: $C_2 \times \{1\} \not\leq D$ as shown above, and for $m=4$, $\{1\} \times V_4 \leq D$
since every element of $V_4$ is even.)

Since $[1, D]$ and $[1, \{1\} \times S_m]$ are isomorphic chains and both
coatoms have the same upper interval $[\cdot, G] \cong \{pt < G\}$,
define $\tau \colon \mathcal{N}(G) \to \mathcal{N}(G)$ by: $\tau$ fixes
$\{1\}$, $C_2 \times \{1\}$, $C_2 \times A_m$, and $G$; swaps the two
coatoms $\{1\} \times S_m \leftrightarrow D$; and, using the interval
isomorphism $[1, \{1\} \times S_m] \cong [1, D]$, swaps each element
below $\{1\} \times S_m$ with its corresponding element below $D$.
In both cases $\{1\} \times A_m$ is fixed: it is the unique minimal element
of both intervals above $\{1\}$ (it lies in $D$ since every $a \in A_m$
satisfies $\sgn(a) = 1$), and for $m = 4$, $\{1\} \times V_4$ is likewise
fixed as the unique second element of both chains.
Every inclusion relation is preserved: inclusions within each interval are
preserved by the interval isomorphism; inclusions involving the fixed elements
$C_2 \times \{1\}$ and $C_2 \times A_m$ are unaffected; and no element of
$[1, \{1\} \times S_m]$ above $\{1\} \times A_m$ is comparable with any
element of $[1, D]$ above $\{1\} \times A_m$ (since $\{1\} \times S_m \not\leq D$
and $D \not\leq \{1\} \times S_m$, as each contains odd-parity elements
the other does not).
Hence $\tau \in \LatAut(G)$, $\tau \neq \id$, and $\tau^2 = \id$.
Since every lattice automorphism fixes $C_2 \times \{1\}$ and the chain
above it, the only remaining freedom is whether to swap $\{1\} \times S_m$
and $D$. Hence $\LatAut(G) = \{\id, \tau\} \cong C_2$.
(For $m = 3$, the interval structure is identical to the case $m \geq 5$.
They both have length-$2$ chains $[1, \{1\} \times S_m]$ and $[1, D]$,
so the same argument applies without modification.)
\end{proof}

\begin{lemma}\label{lem:c2trivial}
$\LatAut(C_2) = 1$.
\end{lemma}
\begin{proof}
$\mathcal{N}(C_2) = \{1, C_2\}$ is a chain of length $1$; any lattice automorphism fixes both elements.
\end{proof}

We can now prove Theorem~\ref{thm:termination}.

\begin{proof}[Proof of Theorem~\ref{thm:termination}]
Let $G_0 = \prod_{k \geq 3} S_k^{a_k}$ be a tower group with parameters
$a_4$ and $B = \sum_{k \neq 4, k \geq 3} a_k$. By Theorem~\ref{thm:product},
\[
  G_1 = \Aut\!\bigl(\mathcal{N}(G_0)\bigr) \;\cong\; S_{a_4} \times S_B.
\]

We carry out a complete case analysis of $G_2 = \Aut(\mathcal{N}(G_1))$ and verify that $G_3 = \Aut(\mathcal{N}(G_2)) = 1$ in each case. Throughout, we adopt the conventions $S_0 = S_1 = 1$.

\medskip
\noindent\textit{Case 1: $G_1 = 1$} (i.e., $a_4 \leq 1$ and $B \leq 1$).
Then $G_2 = 1$ and $G_3 = 1$.

\medskip
\noindent\textit{Case 2: $G_1 = C_2$} (i.e., $(a_4 = 2$ and $B \leq 1)$
or $(a_4 \leq 1$ and $B = 2)$, so that $G_1 \cong S_2 \times S_0 = C_2$
or $G_1 \cong S_0 \times S_2 = C_2$ respectively).
By Lemma~\ref{lem:c2trivial}, $G_2 = 1$ and $G_3 = 1$.

\medskip
\noindent\textit{Case 3: $G_1 = S_n$ for some $n \geq 3$} (i.e.,
$(a_4 \geq 3$ and $B \leq 1)$ or $(a_4 \leq 1$ and $B \geq 3)$,
so that $G_1 \cong S_{a_4} \times S_0 = S_{a_4}$ or
$G_1 \cong S_0 \times S_B = S_B$ respectively).
By Lemma~\ref{lem:single}, $G_2 = 1$ and $G_3 = 1$.

\medskip
\noindent\textit{Case 4: $G_1 = C_2 \times C_2$} (i.e., $a_4 = 2$ and $B = 2$).
By Lemma~\ref{lem:c2sq}, $G_2 = \GL(2, 2) \cong S_3$.
By Lemma~\ref{lem:single}, $G_3 = \Aut(\mathcal{N}(S_3)) = 1$.

\medskip
\noindent\textit{Case 5: $G_1 = C_2 \times S_m$ for some $m \geq 3$} (i.e., one of $a_4, B$ equals $2$ and the other is $\geq 3$).
By Lemma~\ref{lem:c2sm}, $G_2 \cong C_2$.
By Lemma~\ref{lem:c2trivial}, $G_3 = 1$.

\medskip
\noindent\textit{Case 6: $G_1 = S_n \times S_m$ with $n, m \geq 3$} (i.e., $a_4 \geq 3$ and $B \geq 3$).
Here $G_1$ is itself a tower group (with one factor of degree $n$ and one of degree $m$).
Applying Theorem~\ref{thm:product} to $G_1$. Let $a_4' = \mathbf{1}[n = 4] + \mathbf{1}[m = 4]$
and $B' = \mathbf{1}[n \geq 3,\, n \neq 4] + \mathbf{1}[m \geq 3,\, m \neq 4]$.
Then $G_2 = S_{a_4'} \times S_{B'}$. Since $G_1$ has exactly two factors,
$a_4' + B' \leq 2$, so $G_2$ falls into one of Cases 1--5 (with $G_1$ replaced by $G_2$
and so on). Specifically:
\begin{itemize}
  \item If $n = m = 4$: $a_4' = 2$, $B' = 0$, so $G_2 = S_2 = C_2$ (Case 2), $G_3 = 1$.
  \item If exactly one of $n, m$ equals $4$: $a_4' = 1$, $B' = 1$, so $G_2 = S_1 \times S_1 = 1$ (Case 1), $G_3 = 1$.
  \item If neither $n$ nor $m$ equals $4$ (both $\geq 3$): $a_4' = 0$, $B' = 2$, so $G_2 = S_2 = C_2$ (Case 2), $G_3 = 1$.
\end{itemize}
In all sub-cases, $G_3 = 1$.

\medskip
All cases are exhausted. Hence $G_3 = 1$ for every tower group $G_0$.

\medskip
\noindent\textit{Sharpness.} Take $G_0 = S_4^2 \times S_3^2$. Then $a_4 = 2$ and
$B = 2$, giving $G_1 = S_2 \times S_2 = C_2^2$ (Case 4). Then $G_2 = S_3 \neq 1$
and $G_3 = 1$. Since $G_2 = S_3 \neq 1$, the tower requires all three steps:
it does not terminate at step $1$ or $2$.
Hence the bound $G_3 = 1$ is sharp.
\end{proof}

\section{Examples}
\label{sec:examples}

We work through four explicit examples illustrating the Product Formula and the tower.

\begin{example}[$G_0 = S_3^3$]\label{ex:s33}
Here $a_4 = 0$ (no $S_4$ factors) and $B = 3$ (three non-$S_4$ factors, namely
the three copies of $S_3$). The Product Formula gives
\[
  G_1 = \Aut\!\bigl(\mathcal{N}(S_3^3)\bigr) \;\cong\; S_0 \times S_3 = S_3.
\]
The lattice $\mathcal{N}(S_3)$ is the chain
\[
  \{1, A_3, S_3\},
\]
so it has length $2$.
By Lemma~\ref{lem:single}, $G_2 = \Aut(\mathcal{N}(S_3)) = 1$.
The tower terminates in $2$ steps: $S_3^3 \to S_3 \to 1$.
\end{example}

\begin{example}[$G_0 = S_4^3$]\label{ex:s43}
Here $a_4 = 3$ and $B = 0$. The Product Formula gives
\[
  G_1 = \Aut\!\bigl(\mathcal{N}(S_4^3)\bigr) \;\cong\; S_3 \times S_0 = S_3.
\]
By Lemma~\ref{lem:single}, $G_2 = \Aut(\mathcal{N}(S_3)) = 1$. The tower terminates in $2$ steps.
\end{example}

\begin{example}[$G_0 = S_5^2 \times S_3^2$]\label{ex:s52s32}
Here $a_4 = 0$ and $B = 4$ (two copies of $S_5$ and two of $S_3$). The Product Formula gives
\[
  G_1 = \Aut\!\bigl(\mathcal{N}(S_5^2 \times S_3^2)\bigr) \;\cong\; S_0 \times S_4 = S_4.
\]
By Lemma~\ref{lem:single}, $G_2 = \Aut(\mathcal{N}(S_4)) = 1$. The tower terminates in $2$ steps.

\noindent\textit{Verification:} In $\mathcal{N}(S_5^2 \times S_3^2)$, the class $\mathcal{A}$
is empty and $\mathcal{B}$ has four factors ($B = 4$). Every automorphism of this
abstract lattice permutes the four $\mathcal{B}$-factors freely, giving
$\Aut(\mathcal{N}(G_0)) \cong S_4$.
\end{example}

\begin{example}[Sharpness: $G_0 = S_4^2 \times S_3^2$]\label{ex:sharp}
Here $a_4 = 2$ and $B = 2$. We trace the full tower explicitly.

The class $\mathcal{A}$ has $a_4 = 2$ factors (the two copies of $S_4$), and
$\mathcal{B}$ has $B = 2$ factors (the two copies of $S_3$). Every automorphism
of the abstract lattice $\mathcal{N}(G_0)$ permutes the $\mathcal{A}$-factors
among themselves (contributing a factor of $S_2$) and the $\mathcal{B}$-factors
among themselves (another factor of $S_2$), giving
\[
  G_1 = \Aut\!\bigl(\mathcal{N}(S_4^2 \times S_3^2)\bigr) \;\cong\; S_2 \times S_2 = C_2^2.
\]

The lattice $\mathcal{N}(C_2^2)$ has five elements: the identity, three subgroups
of order $2$ (one for each non-identity element of $C_2^2$), and $C_2^2$ itself.
Any automorphism of this abstract lattice permutes the three atoms freely with no
further constraints (see Figure~\ref{fig:nc2sq}), giving
\[
  G_2 = \Aut\!\bigl(\mathcal{N}(C_2^2)\bigr) \;\cong\; S_3.
\]

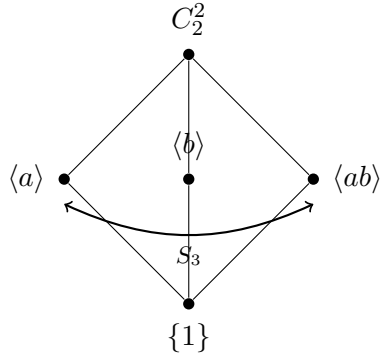
\begin{figure}[ht]
\centering
\begin{tikzpicture}[scale=1.1,
  node/.style={circle, fill=black, inner sep=1.5pt},
  label/.style={font=\small}]
\node[node] (bot) at (0,0) {};
\node[node] (a)   at (-1.5,1.5) {};
\node[node] (b)   at (0,1.5) {};
\node[node] (ab)  at (1.5,1.5) {};
\node[node] (top) at (0,3) {};
\draw (bot) -- (a) -- (top);
\draw (bot) -- (b) -- (top);
\draw (bot) -- (ab) -- (top);
\node[label, below=3pt] at (bot) {$\{1\}$};
\node[label, left=3pt]  at (a)   {$\langle a \rangle$};
\node[label, above=3pt] at (b)   {$\langle b \rangle$};
\node[label, right=3pt] at (ab)  {$\langle ab \rangle$};
\node[label, above=3pt] at (top) {$C_2^2$};
\draw[<->, thick, bend right=25]
  (-1.5,1.2) to node[below, font=\scriptsize]{$S_3$} (1.5,1.2);
\end{tikzpicture}
\caption{The lattice $\mathcal{N}(C_2^2)$: five elements arranged as a
diamond, with three atoms (middle level) permuted by
$\LatAut(C_2^2) \cong S_3$.}
\label{fig:nc2sq}
\end{figure}

Since $\mathcal{N}(S_3) = \{1, A_3, S_3\}$ is a chain of length $2$, its only
automorphism is the identity, giving
\[
  G_3 = \Aut\!\bigl(\mathcal{N}(S_3)\bigr) = 1.
\]

The full tower is therefore
\[
  S_4^2 \times S_3^2
  \;\xrightarrow{\;\Aut(\mathcal{N}(\cdot))\;}
  C_2^2
  \;\xrightarrow{\;\Aut(\mathcal{N}(\cdot))\;}
  S_3
  \;\xrightarrow{\;\Aut(\mathcal{N}(\cdot))\;}
  1,
\]
terminating at step $3$ with $G_2 = S_3 \neq 1$, confirming that the bound
$G_3 = 1$ is sharp.
\end{example}

\section{Conclusion}

We have proved a complete, sharp termination result for the LatAut tower on
tower groups. The Product Formula (Theorem~\ref{thm:product}) compresses the
full lattice-automorphism group $\LatAut(G)$ of any tower group $G$ into a
product of two symmetric groups, determined by counting $S_4$-factors and
non-$S_4$-factors separately. The Termination Theorem (Theorem~\ref{thm:termination})
shows that iterating this compression yields the trivial group within three steps
of the original tower group, and that three steps are genuinely required.

\medskip
\noindent\textit{Scope:} The results apply only to tower groups
$G = \prod_{k \geq 3} S_k^{a_k}$. For general finite groups, neither the
Product Formula nor the step-count bound is established here. The
classification of $\mathcal{N}(G)$ via Goursat's lemma depends heavily on the
normal subgroup structure of the factors $S_k$ (simplicity of $A_k$ for $k \geq 5$,
and the specific structure for $k = 3, 4$), and the argument does not extend
without modification to other group families.

\medskip
\noindent\textit{Obstructions to extension:} Three specific steps in the
argument are tailored to the tower-group setting and fail for more general
direct products. First, the classification of $\mathcal{N}(G)$ via admissible
triples relies on the fact that the only valid Goursat coupling quotient for
pairs of factors in a tower group is $C_2$, the normal subgroup structure of
$S_k$ forces $N_{k,i}/M_{k,i} \in \{1, C_2\}$ for every valid pair, so the sign
subgroup $H \leq \{\pm 1\}^J$ encodes all coupling data. For a direct product
$G = H \times K$ in which $H$ and $K$ share a common quotient $C_p$ for an
odd prime $p$, the Goursat coupling produces elements of $\mathcal{N}(G)$
parameterised by subgroups of $C_p^J$ rather than $\{\pm 1\}^J$, and the
admissibility conditions and product-one characterisation of sign-parity
elements have no direct analogue. Second, the complement-uniqueness argument
in Lemma~\ref{lem:factperm} uses $Z(G) = 1$ to conclude that every element
of $\mathcal{N}(G)$ with a lattice-theoretic complement is a sub-product of
the indecomposable direct factors. For groups with nontrivial centre or with
factors sharing common normal quotients, the correspondence between
lattice-theoretic complements and group-theoretic direct factors breaks down.
Third, the partition of direct factors into classes $\mathcal{A}$ and
$\mathcal{B}$ by the chain length of $\mathcal{N}(S_k)$ at each copy-slot
has no analogue for indecomposable groups whose normal subgroup lattices are
not chains, and the injection $\LatAut(G) \hookrightarrow S_{a_4} \times S_B$
depends essentially on this two-class structure.

\section{Future Directions}

The following problems arise from the methods of this paper.

\begin{enumerate}
\item[\textup{(1)}]
Let $H$ and $K$ be finite groups sharing a common quotient $C_p$ for an
odd prime $p$, so that Goursat coupling across $H \times K$ is parameterised
by subgroups of $C_p^J$ rather than $\{\pm1\}^J$.
Determine $\LatAut(H \times K)$ in this setting and characterise
$\mathcal{N}(H \times K)$ combinatorially.

\item[\textup{(2)}]
The factor-permutation argument of Lemma~\ref{lem:factperm} identifies
direct factors via complement-uniqueness within a single sign-parity cluster
of $\mathcal{N}(G)$. For groups whose normal-subgroup lattice contains
multiple such clusters at distinct levels, determine whether $\LatAut(G)$
admits an analogous description and whether the LatAut tower terminates
in bounded steps.

\item[\textup{(3)}]
The bound $G_3 = 1$ uses the partition of direct factors into classes
$\mathcal{A}$ and $\mathcal{B}$ defined by the chain length of the local
interval $[1, S_k^{(k,i)}]$. For direct products of non-isomorphic
non-abelian simple groups other than symmetric groups, the analogous
partition may have more than two classes. Determine whether the LatAut
tower terminates in such cases and, if so, find a sharp step bound.

\end{enumerate}


\end{document}